\documentclass[12pt,a4paper]{amsart}
\usepackage{mathrsfs}
\usepackage{amsmath}
\usepackage{amsthm}
\usepackage{amsfonts}
\usepackage{latexsym}
\usepackage{latexsym}
\usepackage{graphicx}
\usepackage{hyperref}
\usepackage{amssymb}
\usepackage{subfigure}
\usepackage{hyperref}
\usepackage{fancyhdr}
\allowdisplaybreaks

\begin{document}
\newcommand{\beq}{\begin{equation}}
\newcommand{\eb}{\begin{equation}}
\newcommand{\eneq}{\end{equation}}
\newcommand{\ee}{\end{equation}}
\newtheorem{thm}{Theorem}[section]
\newtheorem{coro}[thm]{Corollary}
\newtheorem{lem}[thm]{Lemma}
\newtheorem{prop}[thm]{Proposition}
\newtheorem{defi}[thm]{Definition}
\newtheorem{rem}[thm]{Remark}
\newtheorem{cl}[thm]{Claim}
\newcommand{\real}{\mathbb{R}}
\newcommand{\calE}{\mathcal{E}}
\title[Multiple nonsmooth solutions for nonconvex variational boundary value problems]{Multiple nonsmooth
solutions for nonconvex variational boundary value problems in
$\mathbb{R}^n$}
\author{Xiaojun Lu$^{1}$ and David Yang Gao$^2$
}
\thanks{Corresponding author: D. Gao}
\thanks{Keywords: Nonconvex variational problems,
nonlinear PDEs, canonical duality-triality theory, analytical
solution}
\thanks{Mathematics Subject Classification: 35J20, 35J60, 74G65,
74S30}
\date{}
\maketitle

\pagestyle{fancy}                   
\lhead{X. Lu and D. Gao} \rhead{Nonconvex variational boundary value problems} 
\begin{center}
1. Department of Mathematics \& Jiangsu Key Laboratory of
Engineering Mechanics, Southeast University, 210096, Nanjing,
China\\
2. Faculty of Science and Technology, Federation University
Australia, Ballarat, VIC 3350, Australia
\end{center}
\vspace{.5cm}
\begin{abstract}
This paper presents a set of complete solutions of a nonconvex
variational problem with a double-well potential. Based on the
canonical duality-triality theory, the associated nonlinear
differential equation with either Dirichlet/Neumann or mixed
boundary conditions can be converted into an algebraic equation,
which can be solved analytically to obtain all solutions in the dual
space. Both global and local extremality conditions are identified
by the triality theory. In the application part, typical mechanical
models with specific sources and boundary conditions in
$\mathbb{R}^2$ are exhibited.
\end{abstract}

\section{Problem formulation and motivations}
Our goal of this paper is to solve the following nonconvex
variational problem
\begin{equation}
(\mathcal{P}_n):
\displaystyle\min_{u\in\mathcal{U}}\Big\{P_n(u):=\int_\Omega
W(\nabla u)dx-\int_\Omega fudx-\int_{\Gamma_t}tu d\Gamma\Big\},
\label{eq-pp}
\end{equation}
where $\Omega\subset\mathbb{R}^n$ is an open, bounded and simply
connected domain with sufficiently smooth boundary $\partial\Omega
=\Gamma=\Gamma_u\cup\Gamma_t$; On $\Gamma_t$, the Neumann force (or
surface traction) $t$ is given, while on $\Gamma_u$, the Dirichlet
boundary condition is prescribed; the source term $f$ can be viewed
as the distributed defects in phase transitions. The nonconvex
function $W$ is a fourth-order polynomial defined as \eb
\displaystyle W(\overrightarrow{y}):={\nu/2}\left({1/
2}|\overrightarrow{y}|^2-\lambda\right )^2,\ee where
$\overrightarrow{y}\in\mathbb{R}^n$, $\nu, \lambda>0$ are given
constants. This function is the so-called double-well potential in
phase transitions, or the Mexican-hat in quantum mechanics. The
double-well potential was first studied by Van der Waals in 1893 for
a compressible fluid whose free energy at constant temperature
depends not only on the density, but also on the density gradient
(see \cite{waals}). Since then, this nonconvex function has been
found extensive applications in nonlinear sciences. For examples, in
phase transitions of Ericksen's bar \cite{E}, or the mathematical
theory of super-conductivity \cite{G-7,gao-yu}, $W$ is the
well-known Landau's second order free energy; each of its local
minimizers represents a possible phase state of the material, while
each local maximizer characterizes the critical conditions that lead
to the phase transitions. In quantum mechanics, if
$\overrightarrow{y}$ represents the Higgs' field strength, then $W$
is the Higgs' potential (see \cite{higgs}). In economics, many
existence results for Nash Equilibrium (NE) points of
non-cooperative games are concerned with this kind of nonconvex
problems. It was discovered in the context of post-buckling analysis
\cite{gao-mrc96,G-2} that the stored potential energy of a large
deformed beam model is also a double-well function, where each
potential-well represents a possible buckled beam state, and the
local maximizer is corresponding to the unbuckled state.
Additionally, the nonconvex function $W$ also plays fundamental
roles in cosmology \cite{kibble}, mathematical economics
\cite{N-B,I-R,E}, chaotic dynamics \cite{ruan-gao-ima13}, finite
deformation mechanics \cite{gao-ima98,G-6,G-7}, and much more (see
review articles \cite{G-2,gao-amma03}).

Our goal is to find  all possible analytical solutions of the
nonconvex variational problem (\ref{eq-pp}). We let $p\in[1,\infty]$
and define a function space $\mathcal{U}$ as \[ \mathcal{U} :=
\Big\{u\in L^p(\Omega)\ |\; \; u(x) = 0 \;\; \text{for}\ \forall x
\in \Gamma_u, \;\; \nabla u\in (L^4(\Omega))^n\Big\}. \] For our
purpose, we restrict $p,q,\beta\in[2,\infty]$ and choose the subset
$\mathcal{U}_0$ of the function space $\mathcal{U}$ as the feasible
function space (let ${1}/{p}+{1}/{p'}=1$ and ${1}/{q}+{1}/{q'}=1$),
\eb\mathcal{U}_0:= \Big\{u\in \mathcal{U}\ |\ \Delta u\in
L^{{2p'}/{(2-p')}}(\Omega), \;\; (u,\frac{\partial u}{\partial{
\overrightarrow{n}}})|_{\Gamma_t}\in L^q(\Gamma_t)\times
L^\beta(\Gamma_t)\Big\}. \ee Here, $\Delta u\in L^\infty(\Omega)$
when $p=2$. It is evident that the above subset is not empty since
$C^m(\overline{\Omega})\subset\mathcal{U}_0$, $m=2,3,\cdots$.
\begin{rem}
Indeed, the above $\mathcal{U}$ is a special Sobolev space equipped
with the norm \[ \|u\|_{\mathcal{U}}:=
\|u\|_{L^p}+(\displaystyle\sum_{i=1}^n\|u_{x_i}\|^2_{L^4})^{{1}/{2}}.
\]
Clearly, $\mathcal{U}$ is a reflexive Banach space when
$p\in(1,\infty)$. If $p=4$, then $\mathcal{U}$ is equivalent to
$W^{1,4}(\Omega)$(see \cite{RA}). Moreover, when $p\in[1,4]$,
$L^4(\Omega)$ is continuously embedded in $L^p(\Omega)$ since
$\Omega$ is bounded. The smoothness of $\Gamma_u\cup\Gamma_t$
assures the regularity on the boundary. For instance, when $p=4$,
$\Gamma_u\cup\Gamma_t\in C^1$, then the trace $\gamma u\in
W^{{3}/{4},4}(\Gamma_u\cup\Gamma_t)$ according to the trace theory
in \cite{LIONS}.
\end{rem}
Therefore, the criticality condition $D P_n(u)=0$ leads to the
following nonlinear partial differential equation in $\Omega$ and
natural boundary condition on $\Gamma_t$:
\begin{equation}\left\{\begin{array}{lll}
\nabla\cdot\overrightarrow{\sigma}(\nabla u)=f &\mbox{in} &\Omega,\\
\\
\overrightarrow{n}\cdot\overrightarrow{\sigma}(\nabla u)=t
&\mbox{on} &\Gamma_t,
\end{array}\right.
\end{equation}
where \[\overrightarrow{\sigma}=(\sigma_1,\cdots,\sigma_n)
=\nu({1}/{2}|\overrightarrow{y}|^2-\lambda)\overrightarrow{y},
\]
and $\overrightarrow{n}$ stands for the unit outward normal vector.
Actually,
\[
\overrightarrow{\sigma}\in\Big\{\overrightarrow{g}\in(L^{4/3}(\Omega))^n\
|\ \nabla\cdot\overrightarrow{g}\in L^{p'}(\Omega),\;\;
\overrightarrow{g}\cdot\overrightarrow{n}\in L^{q'}(\Gamma_t)\Big\}.
\]
Due to the nonconvexity of $W(\overrightarrow{y})$, for given
parameter $\lambda>0$, the source $f \in L^{p'}(\Omega)$ and
boundary term $t\in L^{q'}(\Gamma_t)$, the nonlinear differential
equation (4) may have multiple solutions at each material point
$x\in\Omega$. Since the domain $\Omega$ is continuous, therefore,
the boundary value problem (4) could have infinite number of
solutions. Each of these solutions is a critical point of $P_n(u)$,
i.e. it could be either an extremum or a saddle point of the total
potential. This phenomenon has been verified by Ericksen who proved
that many local solutions are metastable and may have arbitrary
number of phase interfaces. Compared with convex problems, a
fundamentally different issue in nonconvex analysis is that the
solutions of the boundary-value problem is not equivalent to the
associated minimum variational problem. By the fact that the
second-order condition $\delta^2 P_n({\bar u}) \ge 0$ is only a
necessary condition for identifying the global minimal solutions, it
is well-known that traditional direct approaches for solving the
nonconvex variational problem $(\mathcal{P}_n)$ are fundamentally
difficult. Actually, even in the case of finite dimensional space,
many nonconvex global optimization problems are considered to be
NP-hard.

The purpose of this paper is to solve the challenging nonconvex
minimization problem $(\mathcal{P}_n)$ by using the {\em canonical
duality theory.} This is a methodological theory which can be used
for solving a large class of nonconvex/nonsmooth/discrete problems
in multidisciplinary fields, including mathematical physics, global
optimization, computational science, industrial and systems
engineering, etc. \cite{gao-amma03,gao-cace09,gao-sherali}. The
canonical duality theory has been used successfully by D. Y. Gao and
R. W. Ogden for 1D problems in finite deformation mechanics
\cite{G-5} and phase transitions of the Ericksen bar \cite{G-6}.
Their work showed that by using the canonical duality theory, the
nonlinear ordinary differential equations can be converted into
algebraic equations which can be solved completely to obtain all
possible solutions. Both global and local extrema can be identified
by the triality theory. They discovered that for certain given
external loads, the global minimizer is nonsmooth and cannot be
determined by any Newton-type numerical methods.

The rest of the paper is organized as follows. In Section 2, we
apply the nonlinear canonical dual transformation to establish the
perfect dual problem and a pure complementary energy principle for
$(\mathcal{P}_n)$. The triality theory provides both global and
local extremality conditions for the nonconvex problem. A set of
complete solutions for $(\mathcal{P}_n)$ is given and the existence
of analytical solutions of the corresponding boundary value problems
is also discussed. Finally, applications in 2D are illustrated in
Section 3.

\section{Canonical duality techniques and main results}
By the fact that the linear operator $\nabla$ cannot change the
nonconvexity of the double-well function $W(\nabla u)$, instead, we
use the following geometrically nonlinear measure \cite{G-6} \[
\xi:=\Lambda(u)={1/ 2}|\nabla u|^2: \; \mathcal{U}\to\calE \subset
L^2(\Omega),
\]
where $$\calE:=\{\xi\in L^2(\Omega) \; | \;\; \xi \geq 0 \}.$$ Thus,
in terms of this nonlinear measure, the nonconvex function $W(\nabla
u)$ can be written in the so-called canonical form $W(\nabla
u)=U(\Lambda(u))$, where \[ U(\xi):= {\nu/2}( \xi - \lambda)^2,\]
which is a convex function with respect to $\xi$. Therefore, the
canonical dual stress \[\zeta =DU(\xi)=\nu(\xi-\lambda)
\]
is well defined and belongs to \[\mathcal{E}^*:=\Big\{\zeta\in
L^2(\Omega)\ |\ \zeta \geq-\nu\lambda, \;\; \nabla\zeta\in
(L^{{4p'}/{(4-p')}}(\Omega))^n,\;\; \zeta |_{\Gamma_t}\in
L^{{q'\beta}/{(\beta-q')}}(\Gamma_t)\Big\}.
\]
By the Legendre transformation, the complementary energy function
$U^\ast(\zeta)$ can be obtained by
$$
U^\ast(\zeta)=\xi\zeta-U(\xi)={\zeta^2/(2\nu)}+\lambda\zeta.
$$
Replacing $W(\nabla u)=U(\Lambda(u))$ in $(\mathcal{P}_n)$ by
$\Lambda (u)\zeta-U^\ast(\zeta)$, we obtain the Gao-Strang total
complementary energy $\Xi(u,\zeta)$ in the form \begin{equation}
\Xi(u,\zeta):=\displaystyle\int_{\Omega}\{\Lambda(u)\zeta-U^\ast(\zeta)
-fu\}dx-\int_{\Gamma_t}tud\Gamma.
\end{equation}
Next, we introduce the following criticality condition.
\begin{defi}
$(\bar{u}, \bar{\zeta})\in\mathcal{U}_0\times\mathcal{E}$ is said to
be a critical point of $\Xi(u,\zeta)$ if and only if
\begin{equation}
D_{u}\Xi(\bar{u},\bar{\zeta})=0\end{equation} and
\begin{equation}
D_{\zeta}\Xi(\bar{u},\bar{\zeta})=0,
\end{equation}
where $D_{u}, D_{\zeta}$ denote the partial G\^ateaux derivatives,
respectively. \end{defi} For a fixed $\zeta\in\mathcal{E}$, $(6)$
leads to the equilibrium equation
\begin{equation}
\left\{
\begin{array}{cll}
\nabla\cdot(\zeta\nabla\bar{u})+f&=&0\ \mbox{in}\ \Omega,\\
\\
(\zeta\nabla\bar{u})\cdot\overrightarrow{n}&=&t\ \mbox{on}\
\Gamma_t.
\end{array}
\right.
\end{equation}
In particular, according to H\"{o}lder's inequality, $\|\zeta\nabla
u\|_{L^{4/3}(\Omega)}\leq\|\zeta\|_{L^2(\Omega)}\|\nabla
u\|_{L^4(\Omega)}$. While for a fixed $u\in\mathcal{U}_0$, (7) is
consistent with the constitutive law
\begin{equation}
\Lambda(u)=DU^\ast(\bar{\zeta}).
\end{equation}
Next we consider the pure complementary energy functional
\begin{equation}
P^d_n(\zeta):=\Xi(\bar{u},\zeta),
\end{equation}
where $\bar{u}$ is a solution of BVP (8). By applying Green's
formula, $\Xi(u,\zeta)$ can be rewritten as
\begin{equation}\begin{array}{lll}
& & \Xi(u,\zeta) = \displaystyle\int_\Omega\Big\{\Big({1}/{2}|\nabla
u|^2-\lambda\Big)\zeta-U^\ast(\zeta)-fu\Big\}dx-\int_{\Gamma_t} tud\Gamma\\
\\
&=&\displaystyle\int_\Omega\Big\{|\nabla
u|^2\zeta-\lambda\zeta-U^\ast(\zeta)-fu\Big\}dx-\int_\Omega{1}/{2}|\nabla
u|^2\zeta dx-\int_{\Gamma_t} tud\Gamma\\
\\
&=&\displaystyle\underbrace{\int_{\Gamma_t}\Big\{\overrightarrow{\sigma}\cdot\overrightarrow{n}-t\Big\}\
ud\Gamma}_{(I)}-\underbrace{\int_\Omega\Big\{\nabla\cdot
\overrightarrow{\sigma}+f\Big\}udx}_{(II)}-\underbrace{\int_\Omega\Big\{{1}/{2}|\nabla
u|^2\zeta+\lambda\zeta+U^\ast(\zeta)\Big\}dx ,}_{(III)}\\
\\
\end{array}\end{equation}
where $\overrightarrow{\sigma}=\zeta\nabla u $. Therefore, if
$\bar{u}$ solves BVP (8), then the pure complementary energy
functional is in fact
\begin{eqnarray}\label{p-nonconvex-10}
P^d_n(\zeta)=-{1/ 2}\int_{\Omega}\Big({|\overrightarrow{\sigma}|^2/
\zeta}+2\lambda\zeta +{\zeta^2/\nu}\Big)dx,
\end{eqnarray}
where $\overrightarrow{\sigma}$ is a solution of the BVP (4). From
the constitutive principle, it is clear that
$|\overrightarrow{\sigma}|^2=o(\zeta)$. The variation of $P_n^d$
with respect to $\zeta$ leads to the dual algebraic equation (DAE),
namely,
\begin{equation}
|\overrightarrow{\sigma}|^2=2\zeta^2(\lambda+{\zeta/\nu}).
\end{equation}
For given parameters $\nu$, $\lambda$ and $\overrightarrow{\sigma}$,
the three complex solutions of the cubic DAE (13) are listed below,
\begin{equation}\zeta_1={1}/{3}\left(-\nu \lambda
+{\sqrt[3]{4}\nu^2\lambda^2}{\omega^{-1}(\nu,\lambda,\overrightarrow{\sigma})}+{\omega(\nu,\lambda,\overrightarrow{\sigma})}4^{-1/3}\right),
\end{equation}
\begin{equation}\zeta_2=-{\nu\lambda}/{3}-3^{-1}\cdot 2^{-1/3}{\left(1-i
\sqrt{3}\right)\nu^2\lambda^2}{
\omega^{-1}(\nu,\lambda,\overrightarrow{\sigma})}-6^{-1}\cdot4^{-1/3}{\left(1+i\sqrt{3}\right)
\omega(\nu,\lambda,\overrightarrow{\sigma})},
\end{equation}
\begin{equation}\zeta_3=-{\nu\lambda}/{3}-3^{-1}\cdot 2^{-1/3}{\left(1+i\sqrt{3}\right)\nu^2
\lambda^2}{\omega^{-1}(\nu,\lambda,\overrightarrow{\sigma})}-6^{-1}\cdot4^{-1/3}{\left(1-i
\sqrt{3}\right)\omega(\nu,\lambda,\overrightarrow{\sigma})},\end{equation}
where
$$\omega(\nu,\lambda,\overrightarrow{\sigma}):=\left(-4\nu^3\lambda^3+27\nu|\overrightarrow{\sigma}|^2+3
\sqrt{3}\sqrt{-8\nu^4\lambda^3|\overrightarrow{\sigma}|^2+27\nu^2|\overrightarrow{\sigma}|^4}\right)^{1/3}.$$
\begin{lem}
From (13)-(16), we know that $|\overrightarrow{\sigma}|^2$ has a
maximum $8\lambda^3\nu^2/27$ at $\zeta=-2\lambda\nu/3$ and minimum 0
at 0. If
$|\overrightarrow{\sigma}|^2\in(8\lambda^3\nu^2/27,\infty)$, then
there exists only one real root $\zeta>0$ of the polynomial (13). If
$|\overrightarrow{\sigma}|^2\in(0,8\lambda^3\nu^2/27)$, then there
exist three real roots $\zeta_1>\zeta_2>\zeta_3$. While when
$|\overrightarrow{\sigma}|^2=8\lambda^3\nu^2/27$, there exist two
real roots.
\end{lem}
\begin{proof} It suffices to prove the fact
$\zeta_1>\zeta_2>\zeta_3$ when
$|\overrightarrow{\sigma}|^2\in(0,8\lambda^3\nu^2/27)$. Actually,
$\omega\bar{\omega}=\sqrt[3]{16}\nu^2\lambda^2.$ From complex
analysis, it is reasonable to set
$\omega=\sqrt[3]{4}\nu\lambda(\cos\theta+i\sin\theta)$,
$\theta\in(0,\pi/3)$. Through simple calculation, one knows
immediately that $$\zeta_1=1/3\nu\lambda(2\cos\theta-1)>0;$$
$$-2\nu\lambda/3<\zeta_2=-\nu\lambda/3(1+\cos\theta-\sqrt{3}\sin\theta)<0;$$
$$-\nu\lambda<\zeta_3=-\nu\lambda/3(1+\cos\theta+\sqrt{3}\sin\theta)<-2\nu\lambda/3.$$ Our
proof is concluded.
\end{proof}
By comparing (4) with (8), we deduce that, for $i,j=1, \dots ,n$, in
order to give an integral form of the solution $u$, the following
compatibility condition has to be satisfied
\begin{equation}
\Phi_\zeta(\sigma_i,\sigma_j):=\left|\begin{array}{cc}
\partial_{x_i}&
\partial_{x_j}\\
\\
\sigma_i\zeta^{-1}&\sigma_j\zeta^{-1}
\end{array}\right|=0.
\end{equation}
Let us define the subregion $\mathcal{S}$ as
\begin{equation}
\displaystyle\mathcal{S}:=\Big\{x\in\overline{\Omega}\ | \
\Phi_\zeta(\sigma_i,\sigma_j)=0, \ i,j=1, \dots,n \Big\}.
\end{equation}
Evidently, the compatibility condition (17) guarantees the path
independence of the integral for $\overrightarrow{\sigma}\zeta^{-1}$
in $\mathcal{S}$. By replacing $\overrightarrow{\sigma}$ in (4) by
$\nu({1}/{2}|\nabla u|^2-\lambda)\nabla u$, then $(\mathcal{P}_n)$
is equivalent to the following BVP,
\begin{equation} \left\{
\begin{array}{cll}
\displaystyle
 \nabla\cdot(\nu({1}/{2}|\nabla u|^2-\lambda)\nabla u)+f&=&0\ {\mbox{in}}\ \Omega,\\
 \\
 (\nu({1}/{2}|\nabla u|^2-\lambda)\nabla u)\cdot\overrightarrow{n}&=&t\ {\mbox{on}}\
 \Gamma_t.
\end{array}
\right.
\end{equation}
In $\mathcal{S}$, the analytical solutions of BVP (19) can be given
by the path integral
\begin{equation}
u(x)=\int^{x}_{x_0}\overrightarrow{\sigma}\zeta^{-1}d{s}+u(x_{0}),
\end{equation}
where $x,\ x_{0} \in\mathcal{S}$. Summarizing the above discussion,
we obtain the theorem below.
\begin{thm}
For a given source $f(x)$ and boundary condition $t$ such that
$\overrightarrow{\sigma}(x)$ is determined by BVP (4), then DAE (13)
has at most three real roots $\bar{\zeta}_i(x)$, $i=1,2,3$, given by
(14)-(16) and ordered as
\begin{equation}
\bar{\zeta}_1(x)\geq0\geq\bar{\zeta}_2(x)\geq-2\nu\lambda/3\geq\bar{\zeta}_3(x)\geq-\nu\lambda.
\end{equation}
For $i=1,2,3$, the functions defined in $\mathcal{S}$ by
\begin{equation} \bar{u}_i(x)=\int^{x}_{x_0}\overrightarrow{\sigma}(s)\bar{\zeta}^{-1}_i(s)d{s}+u(x_{0})
\end{equation}
are solutions of BVP (19). Furthermore,
\begin{equation}
P_n(\bar{u}_i)=P^d_n(\bar{\zeta}_i), i=1,2,3.
\end{equation}
\end{thm}
\begin{proof} The relation (23) is obtained by direct calculation
from the representations of $P_n(u)$ and $P^d_n(\zeta)$,
respectively.
\end{proof}
Theorem 2.3 demonstrates that the pure complementary energy
functional $P_n^d(\zeta)$ is canonically dual to the total potential
energy functional $P_n(u)$. The equation (23) indicates there is no
duality gap between the primal and dual variational problems. In the
following, we apply the triality theory to obtain the extremality
conditions for these critical points.
\begin{thm}
Suppose that the source term $f$ and boundary condition $t$ are
given and $\overrightarrow{\sigma}(x)$ satisfies the divergence
equation (4). Then, if
$|\overrightarrow{\sigma}(x)|^2\in(8\lambda^3\nu^2/27,\infty)$,
$\forall\ x\in\mathcal{S}$, then DAE (13) has a unique solution
$\bar{\zeta}>0$, which is a global maximizer of $P_n^d(\zeta)$ over
$\mathcal{E}$, and the corresponding solution $\bar{u}$ in the form
of (20) is a global minimizer of $P_n(u)$ over $\mathcal{U}_0$,
\begin{equation}
P_n(\bar{u})=\displaystyle\min_{u\in\mathcal{U}_0}P_n(u)=\displaystyle\max_{\zeta\in\mathcal{E}}P_n^d(\zeta)=P_n^d(\bar{\zeta}).
\end{equation}
If $|\overrightarrow{\sigma}(x)|^2\in(0,8\lambda^3\nu^2/27)$,
$\forall\ x\in\mathcal{S}$, then DAE (13) has three real roots
ordered as in Theorem 2.3. Furthermore, $\bar{\zeta}_1$ is a local
maximizer of $P_n^d(\zeta)$ over $\zeta>0$, the corresponding
solution $\bar{u}_1$ is a local minimizer of $P_n(u)$ over
$\mathcal{U}_1$,
\begin{equation}
P_n(\bar{u}_1)=\displaystyle\min_{u\in\mathcal{U}_1}P_n(u)=\displaystyle\max_{\zeta>0}P_n^d(\zeta)=P_n^d(\bar{\zeta}_1),
\end{equation}
 where $\mathcal{U}_1$ is a neighborhood of
$\bar{u}_1$. While for the local maximizer $\bar{\zeta}_3$, the
corresponding solution $\bar{u}_3$ is a local maximizer of $P_n(u)$,
\begin{equation}
P_n(\bar{u}_3)=\displaystyle\max_{u\in\mathcal{U}_3}P_n(u)=\displaystyle\max_{-\nu\lambda<\zeta<-2\nu\lambda/3}P_n^d(\zeta)=P_n^d(\bar{\zeta}_3),
\end{equation}
where $\mathcal{U}_3$ is a neighborhood of $\bar{u}_3$. As for
$\bar{\zeta}_2$, in the case of 1D, the corresponding solution
$\bar{u}_2$ is a local minimizer of $P_n(u)$,
\begin{equation}
P_n(\bar{u}_2)=\displaystyle\min_{u\in\mathcal{U}_2}P_n(u)=\displaystyle\min_{-2\nu\lambda/3<\zeta<0}P_n^d(\zeta)=P_n^d(\bar{\zeta}_2),
\end{equation}
where $\mathcal{U}_2$ is a neighborhood of $\bar{u}_2$. It is worth
noticing that, when $n\geq2$, $\bar{u}_2$ is not necessarily a local
minimizer.
\end{thm}
\begin{proof}
First, we recall the second variation formula for both $P_n(u)$ and
$P_n^d(\zeta)$. On the one hand, for $\forall\
\varsigma\in\mathcal{U}_1:=\Big\{u\in\mathcal{U}_0\Big|\ \nabla
u\neq0\Big\},$
\begin{eqnarray*}
\delta^2_\varsigma P_n(u) &=&\int_\Omega\frac{d^2}{dt^2}\Big\{\nu/2\Big(1/2|\nabla(u+t\varsigma)|^2-\lambda\Big)^2\Big\}\Big|_{t=0}dx \\
&=& \nu\int_\Omega\Big\{|\nabla
u\cdot\nabla\varsigma|^2+\Big(1/2|\nabla
u|^2-\lambda\Big)|\nabla\varsigma|^2\Big\}dx.
\end{eqnarray*}
On the other hand, for $\forall\ \eta\neq0\in\mathcal{E}$,
\begin{eqnarray*}\delta^2_\eta P_n^d(\zeta)&=&-{1}/{2}\int_\Omega\frac{d^2}{dt^2}\Big\{{|\overrightarrow{\sigma}|^2}/{(\zeta+t\eta)}+2\lambda(\zeta+t\eta)+{(\zeta+t\eta)^2}/{\nu}\Big\}\Big|_{t=0}dx\\
&=&-\int_\Omega\Big\{{|\overrightarrow{\sigma}|^2}/{\zeta^3}+{1}/{\nu}\Big\}\eta^2dx.\end{eqnarray*}
If $\zeta>0$, according to the definition of $\zeta=\nu(1/2|\nabla
u|^2-\lambda),$ one knows immediately that
$$\delta^2_\varsigma P_n(u)>0,\ \ \delta^2_\eta P_n^d(\zeta)<0.$$ Then (24) and (25) are
concluded. Actually, when
$|\overrightarrow{\sigma}|^2\in(8\lambda^3\nu^2/27,\infty)$, then by
applying the definition of $\overrightarrow{\sigma}$, one has
$$\lambda^3<27/8(1/2|\nabla u|^2-\lambda)^2|\nabla u|^2.$$ Solving
this inequality, we obtain $\lambda\in(0,3/8|\nabla u|^2).$ Now we
consider the negative $\bar{\zeta}_i(x)$, $i=2,3$. When
$\zeta<-2/3\nu\lambda$, then $\delta^2_\eta P_n^d(\zeta)<0$ and
$\lambda>3/2|\nabla u|^2$. In this case,
\begin{eqnarray*}
\delta^2_\varsigma P_n(u) & \leq &  \nu\int_\Omega\Big\{|\nabla
u|^2|\nabla\varsigma|^2+\Big(1/2|\nabla
u|^2-\lambda\Big)|\nabla\varsigma|^2\Big\}dx \\
& = & \nu\int_\Omega\Big(3/2|\nabla
u|^2-\lambda\Big)|\nabla\varsigma|^2dx<0.
\end{eqnarray*}
Then (26) is proved. It remains to consider (27). For
$\zeta>-\sqrt[3]{|\overrightarrow{\sigma}|^2\nu}$, the second
variation $\delta^2_\eta P_n^d(\zeta)>0$. In addition,
$-2\nu\lambda/3<\zeta<0$ indicates
$\lambda\in(1/2|u_x|^2,3/2|u_x|^2)$. In fact,
$-2/3\nu\lambda<\zeta<0$ gives $\delta^2_\eta P_n^d(\zeta)>0$, which
indicates that $\bar{u}_2$ is a local minimizer.
\end{proof}
\section{Applications}
Now we apply Theorem 2.4 to practical problems in 2D domain. Let us
consider the open annulus $\Omega=\mathbb{A}_{R_1}^{R_2}$ bounded by
two concentric circles with radii $R_1$ and $R_2$($R_1<R_2$). Let
$f$ and $t$ be given, $\Gamma_t:=\Gamma_1\cup\Gamma_2$,
$$
f=-r:=-\sqrt{x^2 + y^2},\ \ t={R_1^2/ 3}\ \text{on}\ \Gamma_1,\ \
t=-{R_2^2/ 3}\ \text{on}\ \Gamma_2.$$ It is evident that $f\in
L^\infty(\Omega)$ and $t\in L^\infty(\Gamma_1\cup\Gamma_2)$. In the
case of $\nu=1,\lambda=1$, the primal problem is of the form
$$(\mathcal{P}_2):\ \ \ \
\displaystyle\min_{u\in\mathcal{U}_0}\Big\{P_2(u)=\int_{\Omega}\Bigl[{1/2}\Big({1/2}|\nabla
u|^2-1\Big)^2-ru\Bigr] dx+\int_{\Gamma_2}{uR_2^2/
3}d\Gamma-\int_{\Gamma_1}{uR_1^2/ 3} d\Gamma\Big\}.
$$

\begin{figure*}[htbp]
  \centering
  \subfigure[$\bar{\zeta}_1$ with respect to $r$]{
    \label{fig:subfig:a} 
    \includegraphics[width=1.8in]{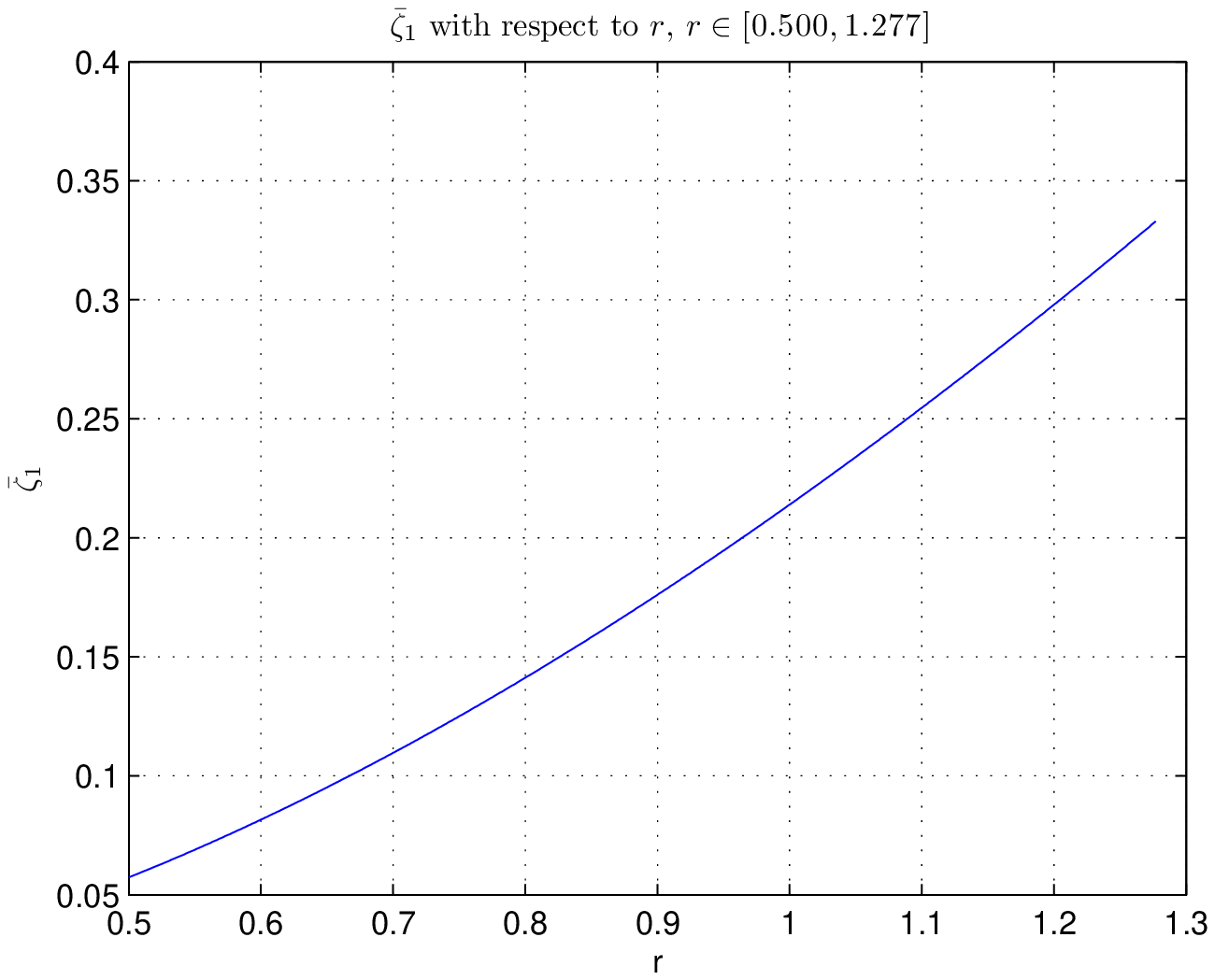}}
  \hspace{0.1in}
\subfigure[$\bar{\zeta}_1$ with respect to $x$ and $y$]{
    \label{fig:subfig:a} 
    \includegraphics[width=1.8in]{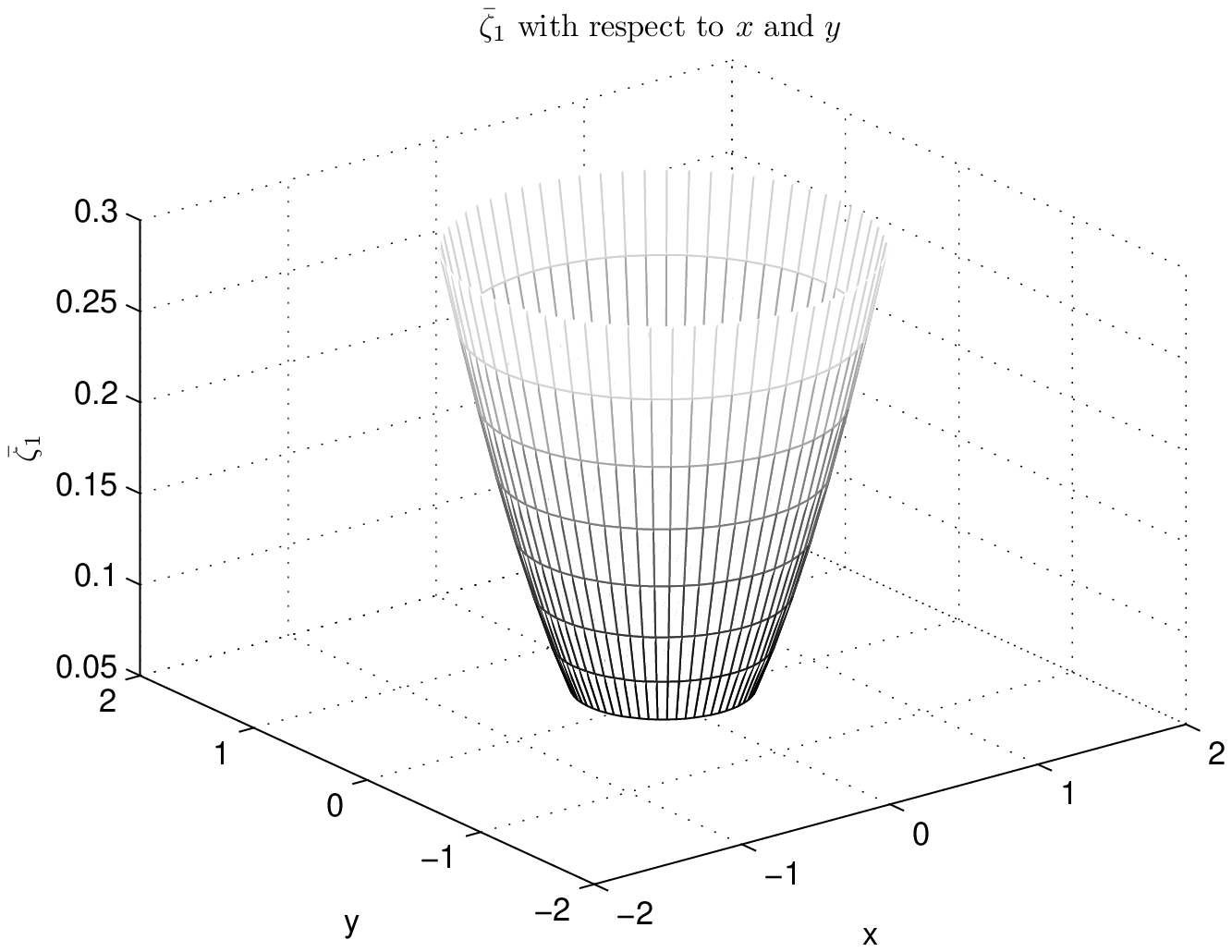}}
    \hspace{0.1in}
\subfigure[$\bar{u}_1(r)-\bar{u}_1(0.500)$ with respect to $r$]{
    \label{fig:subfig:a} 
    \includegraphics[width=1.8in]{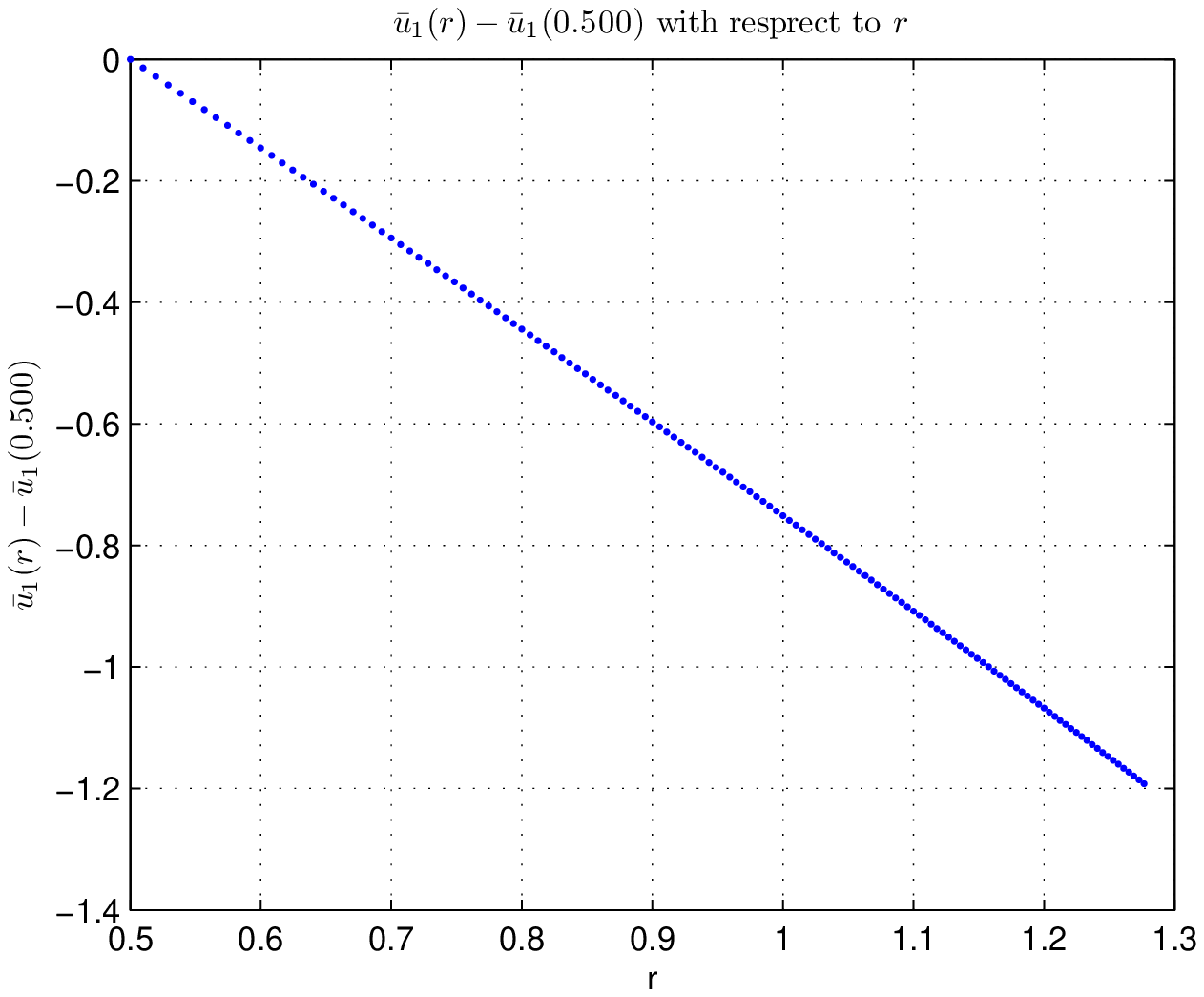}}
  \caption{Dual solution $\bar{\zeta}_1$ and primal solution $\bar{u}_1$, $r\in[0.500,1.277]$}
  \label{fig:subfig} 
\end{figure*}
\begin{figure*}[htbp]
  \centering
  \subfigure[$\bar{\zeta}_2$ with respect to $r$]{
    \label{fig:subfig:a} 
    \includegraphics[width=1.8in]{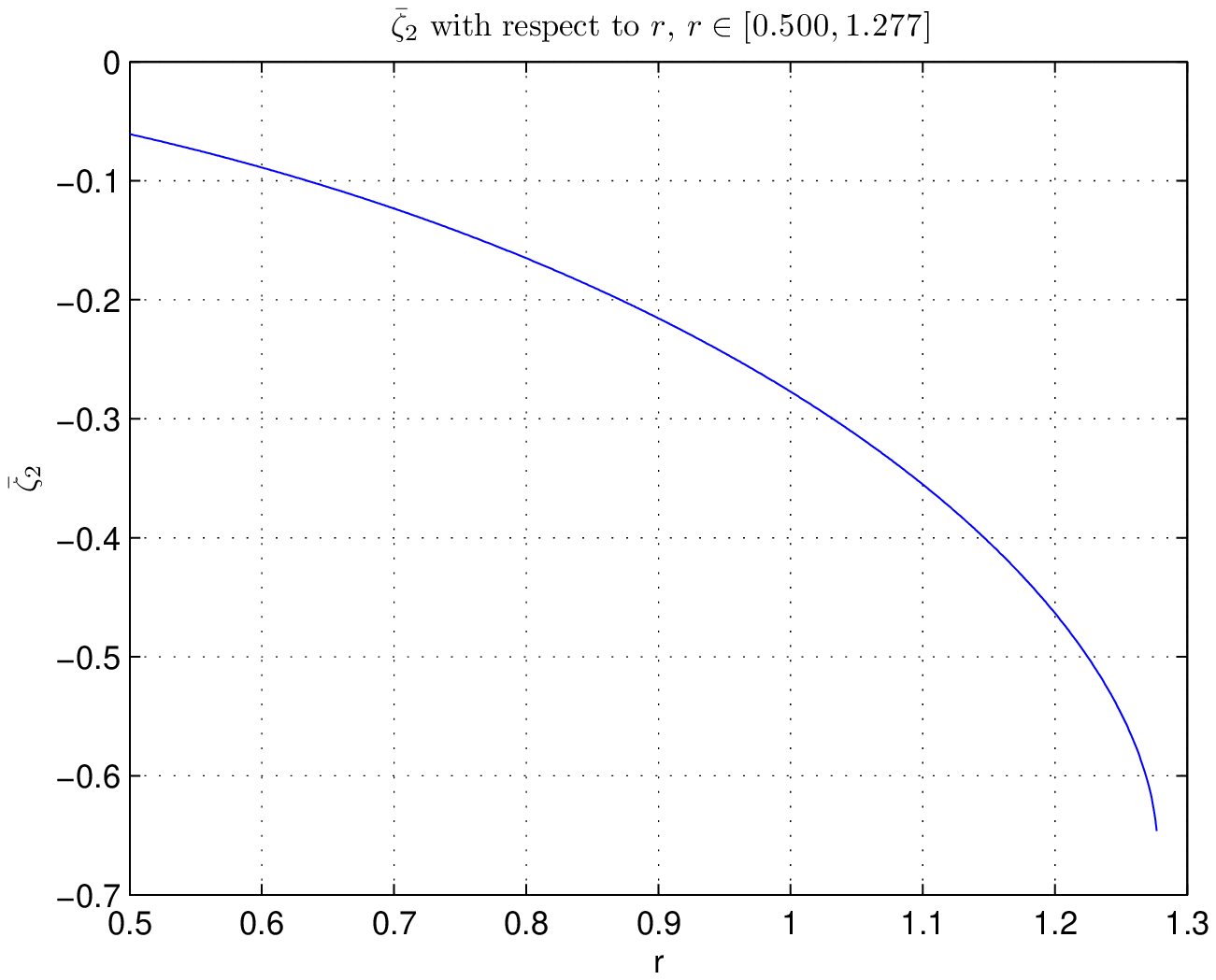}}
  \hspace{0.1in}
\subfigure[$\bar{\zeta}_2$ with respect to $x$ and $y$]{
    \label{fig:subfig:a} 
    \includegraphics[width=1.8in]{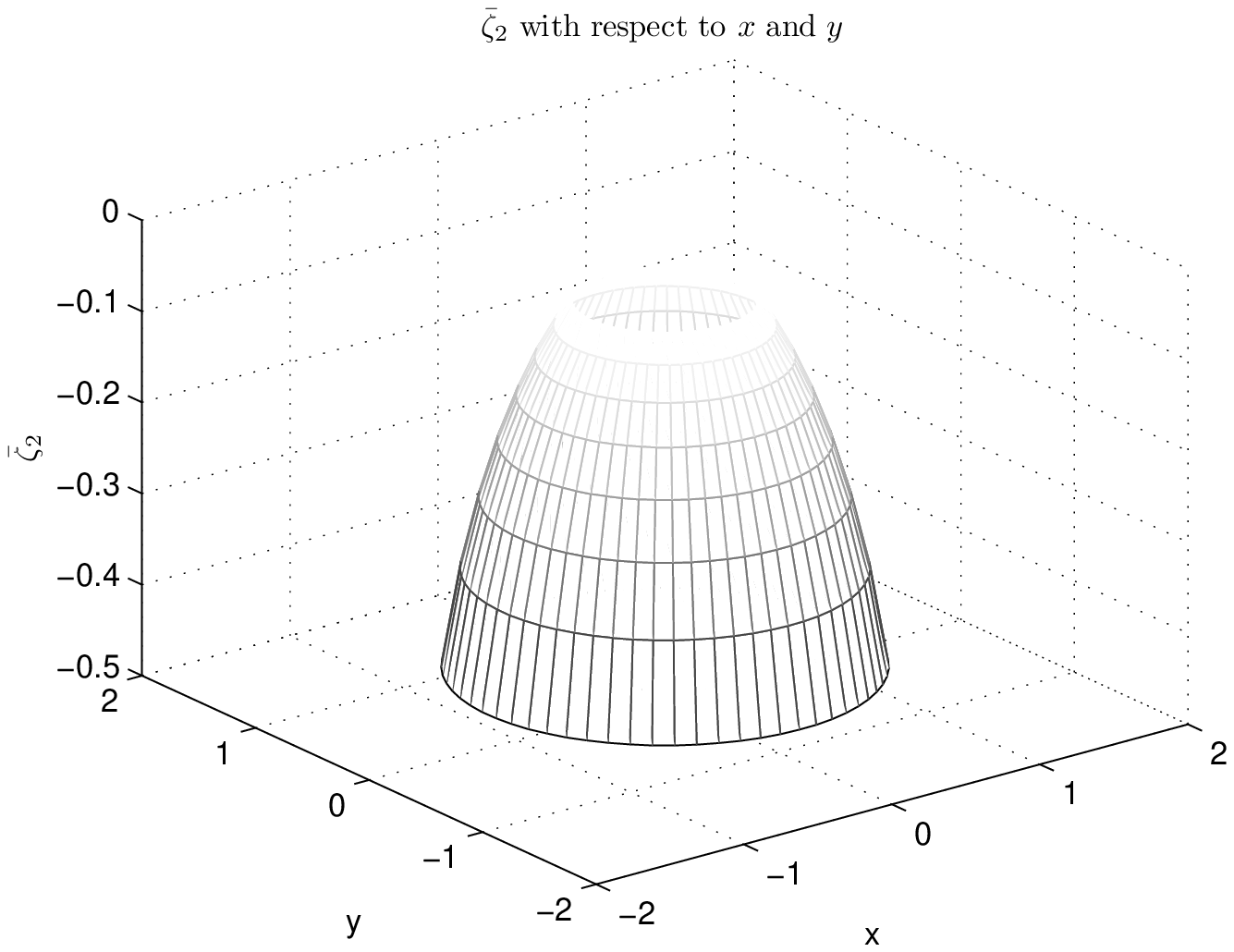}}
    \hspace{0.1in}
\subfigure[$\bar{u}_2(r)-\bar{u}_2(0.500)$ with respect to $r$]{
    \label{fig:subfig:a} 
    \includegraphics[width=1.8in]{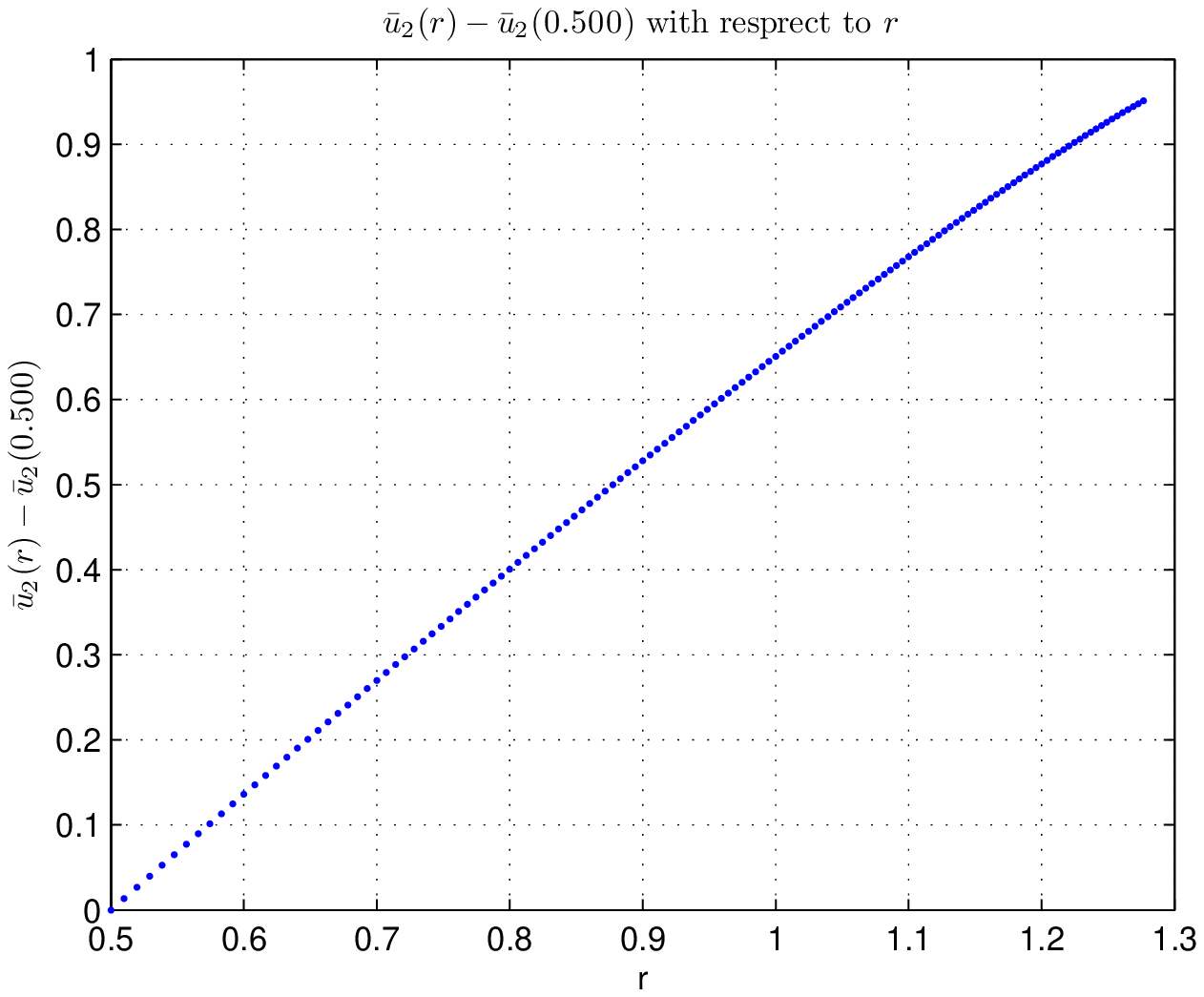}}
  \caption{Dual solution $\bar{\zeta}_2$ and primal solution $\bar{u}_2$, $r\in[0.500,1.277]$}
  \label{fig:subfig} 
\end{figure*}
\begin{figure*}[htbp]
  \centering
  \subfigure[$\bar{\zeta}_3$ with respect to $r$]{
    \label{fig:subfig:a} 
    \includegraphics[width=1.8in]{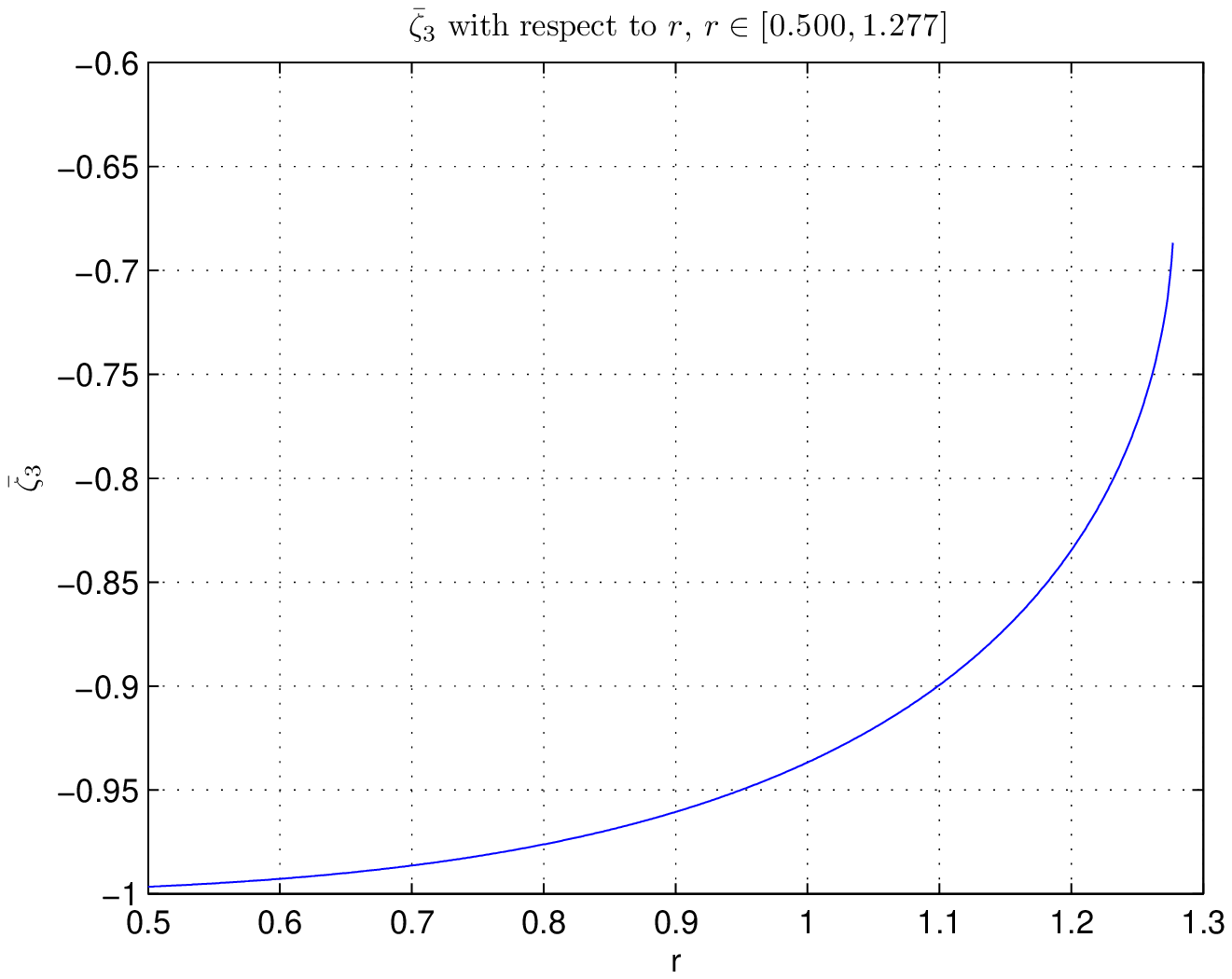}}
  \hspace{0.1in}
\subfigure[$\bar{\zeta}_3$ with respect to $x$ and $y$]{
    \label{fig:subfig:a} 
    \includegraphics[width=1.8in]{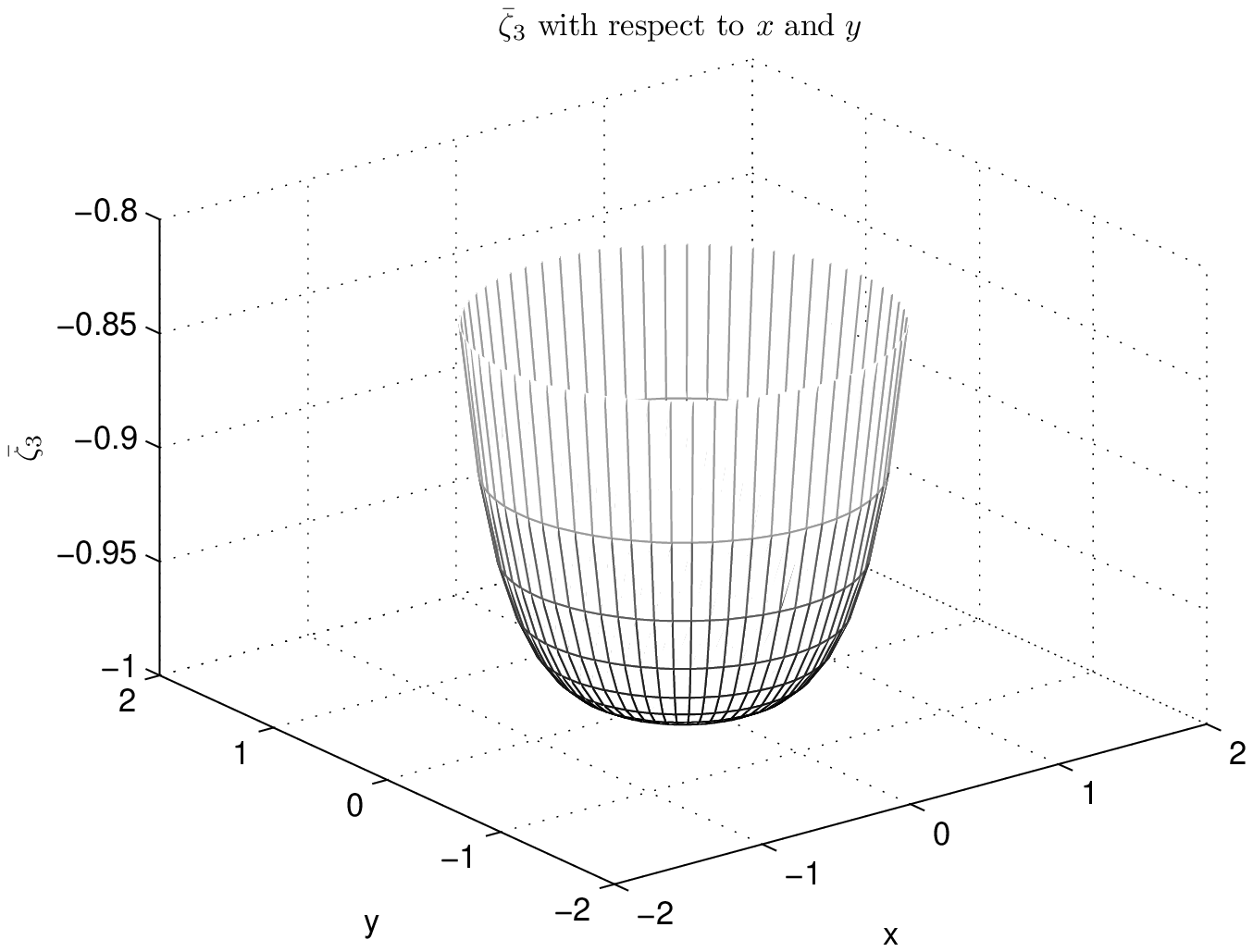}}
    \hspace{0.1in}
\subfigure[$\bar{u}_3(r)-\bar{u}_3(0.500)$ with respect to $r$]{
    \label{fig:subfig:a} 
    \includegraphics[width=1.8in]{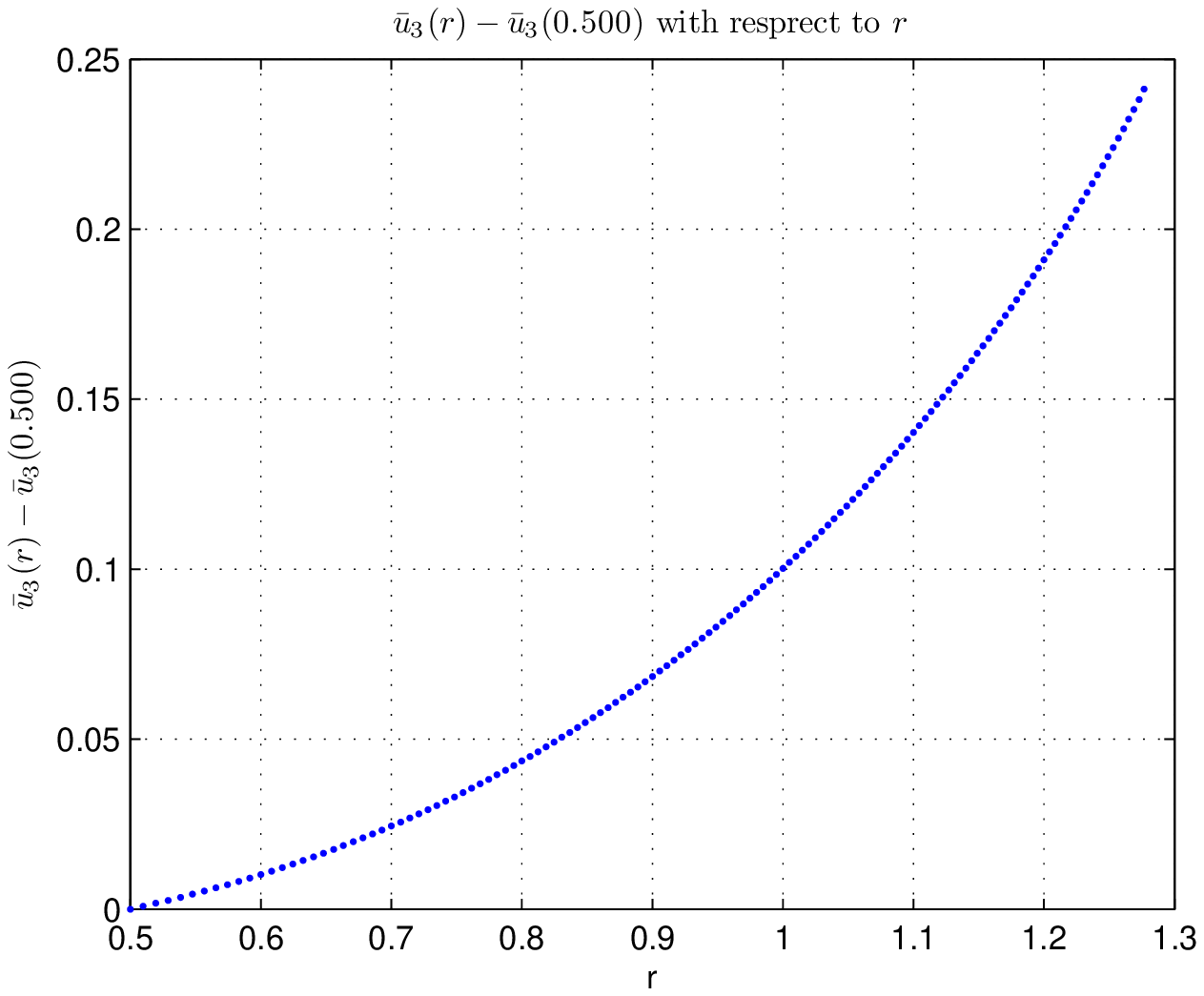}}
  \caption{Dual solution $\bar{\zeta}_3$ and primal solution $\bar{u}_3$, $r\in[0.500,1.277]$}
  \label{fig:subfig} 
\end{figure*}
We consider the radially symmetric solution of (19). Without any
confusion, we denote the radially symmetric functions $u(r):=u(x,y)$
and $\zeta(r):=\zeta(x,y)$. For (4), we have the unique solution in
the form of
$\overrightarrow{\sigma}=\displaystyle(-{rx/3},-{ry/3}),$
$r\in(R_1,R_2)$. Then the singular algebraic polynomial (13) is
given as
\begin{equation}
{r^4/9}=2\zeta^2(1+\zeta). \end{equation} From Theorem 2.3, one
knows immediately that when $r^4\in(8/3,\infty)$, (28) has a unique
real root; while when $r^4\in(0,8/3)$, (28) has three real roots(see
Figures 1-3(a)(b)). For instance,
\begin{itemize} \item if we set $r=2$, then
$\bar{\zeta}_1=0.719078$, $\bar{\zeta}_2=-0.859539-0.705226i$,
$\bar{\zeta}_3=-0.859539+0.705226i$; \item if we set $r=1$, then
$\bar{\zeta}_1=0.213928$, $\bar{\zeta}_2=-0.277249$,
$\bar{\zeta}_3=-0.936679$;
\item if we set $r=0.5$, then
$\bar{\zeta}_1=0.0573064$, $\bar{\zeta}_2=-0.0608031$,
$\bar{\zeta}_3=-0.996503$.
\end{itemize}
Since $\bar{\zeta}_i$ is radially symmetric,  as a result,
\begin{equation}{\overrightarrow{\sigma}/\bar{\zeta}_i}=(-{rx/
{(3\bar{\zeta}_i)}},-{ry/ {(3\bar{\zeta}_i)}})\end{equation}
satisfies the compatibility condition (17) on $\overline{\Omega}$,
$i=1,2,3$. According to Theorem 2.4, we have corresponding
analytical solutions of the form for $i=1,2,3$, respectively,
\begin{equation}
\bar{u}_i(x,y)=-\int_{(x_0,y_0)}^{(x_0,y)}\frac{y\sqrt{x_0^2+y^2}}{3\bar{\zeta}_i(\sqrt{x_0^2+y^2})}dy-
\int_{(x_0,y)}^{(x,y)}\frac{x\sqrt{x^2+y^2}}{3\bar{\zeta}_i(\sqrt{x^2+y^2})}dx+\bar{u}_i(x_0,y_0),
\end{equation}
where $(x_0,y_0),(x,y)\ \text{on}\ \overline{\Omega}$(see Figures
1-3(c)).\\
\\
{\bf Acknowledgement}: The main results in this paper were obtained
during a research collaboration at the Federation University
Australia in August, 2015. The first author wishes to thank
Professor David Y. Gao for his hospitality and financial support.
This project is partially supported by US Air Force Office of
Scientific Research (AFOSR FA9550-10-1-0487), Natural Science
Foundation of Jiangsu Province (BK 20130598), National Natural
Science Foundation of China (NSFC 71273048, 71473036, 11471072), the
Scientific Research Foundation for the Returned Overseas Chinese
Scholars, Fundamental Research Funds for the Central Universities on
the Field Research of Commercialization of Marriage between China
and Vietnam (No. 2014B15214). This work is also supported by Open
Research Fund Program of Jiangsu Key Laboratory of Engineering
Mechanics, Southeast University (LEM16B06). In particular, the
authors also express their deep gratitude to the referees for their
careful reading and useful remarks.

\end{document}